\documentclass[12pt,a4paper]{article}
\usepackage{amsmath, enumerate, fullpage,wasysym, bbm,bm,color}
\usepackage{amsthm}
\usepackage{amssymb}
\usepackage{amsbsy}
\usepackage{amsfonts}
\usepackage{amstext}
\usepackage{amscd}

\numberwithin{equation}{section}
\newtheorem{theorem}{Theorem}[section]

\newtheorem{definition}[theorem]{Definition}
\newtheorem{example}[theorem]{Example}

\newtheorem{proposition}[theorem]{Proposition}
\newtheorem{remark}[theorem]{Remark}

\newcommand{\R}{\mathbb{R}}
\newcommand{\comp}{\mathbb{C}}
\newcommand{\tor}{\mathbb{T}}
\newcommand{\disc}{\mathbb{D}}
\newcommand{\calt}{\circlearrowright}

\newcommand{\M}{\mathcal{M}}
\newcommand{\MM}{\mathbf{M}}
\newcommand{\U}{\mathbf{U}}
\newcommand{\B}{\mathbb{B}}
\newcommand{\D}{\mathbf{D}}
\newcommand{\MP}{\bm{\pi}}
\newcommand{\Beta}{\bm{\beta}}
\newcommand{\bmrho}{\bm{\rho}}
\newcommand{\bmsigma}{\bm{\sigma}}

\newcommand{\iu}{\mathcal{UI}}

\def\mrhd{\kern0.2em\rule{0.035em}{0.52em}\kern-.35em\gtrdot \kern-.2em}
\def\submrhd{\mathop{\kern0.0em\lower0.05ex\hbox{\rule{0.03em}{0.39em}}\kern-0.08em\gtrdot \kern0.0em}}
\def\dismrhd{\mathop{{\kern0.1em\lower0.07ex\hbox{\rule{0.035em}{0.57em}}\kern-.1em\gtrdot \kern0.02em}}}

\newcommand{\utimes}{\kern0.05em\buildrel{\times}\over{\rule{0em}{0.004em}} \kern-0.9em\cup \kern0.2em}
\newcommand{\putimes}{\mathop{\kern0.05em\buildrel{\times}\over{\rule{0em}{0.0em}} \kern-0.9em\cup \kern0em}}
\newcommand{\hutimes}{\mathop{\kern0.02em\buildrel{\times}\over{\rule{0em}{0.0em}} \kern-0.48em\cup \kern0em}}
\newcommand{\sutimes}{\mathrel{\kern0em \buildrel{\mathsf{x}}\over{\rule{0em}{0.0em}} \kern-0.35em\cup \kern-0.0em}}

\begin{document}
\title{Free Subordination and Belinschi-Nica Semigroup
}
\author{Octavio Arizmendi \\ Department of Probability and Statistics, CIMAT, \\ Guanajuato, Mexico \\  Email:
octavius@cimat.mx  \\ \\ Takahiro Hasebe\footnote{Supported by European Commision, Marie Curie Actions -- International Incoming Fellowships (Project 328112 ICNCP) at University of Franche-Comt\'e; supported also by JSPS, Global COE program ``Fostering top leaders in mathematics -- broadening the core and exploring new ground'' at Kyoto university.}\\ 
Department of Mathematics, Hokkaido University \\ Kita 10, Nishi 8, Kitaku, Sapporo 060-0810, Japan. \\ Email:
thasebe@math.sci.hokudai.ac.jp } 
\date{\today}

\maketitle
\begin{abstract}
We realize the Belinschi-Nica semigroup of homomorphisms as a free multiplicative subordination. This realization allows to define more general semigroups of homomorphisms with respect to free multiplicative convolution. For these semigroups we show that a differential equation holds, generalizing the complex Burgers equation. We give examples of free multiplicative subordination and find a relation to the Markov-Krein transform, Boolean stable laws and monotone stable laws. A similar idea works for additive subordination, and in particular we study the free additive subordination associated to the Cauchy distribution and show that it is a homomorphism with respect to monotone, Boolean and free additive convolutions. 
\end{abstract}

Mathematics Subject Classification 2010: 46L54

Keywords: Free subordination, Boolean stable law, monotone stable law, Markov-Krein transform

\tableofcontents

\section{Introduction}

%{\color{red} Revisions: abstract}

In \cite{BN08}, Belinschi and Nica defined, for each $t>0,$  the map  
\[
\mathbb{B}_t(\mu) = \Big(\mu^{\boxplus (1+t)}\Big) ^{\uplus\frac{1}{1+t}},\qquad \mu\in\mathcal{P}(\R).  
\]
The family $\{\mathbb{B}_t\}_{t\geq0}$  is a composition semigroup of homomorphisms in the sense that it satisfies
$$
\mathbb{B}_t(\mu_1\boxtimes\mu_2)=\mathbb{B}_t(\mu_1)\boxtimes\mathbb{B}_t(\mu_2), \qquad\mu_1,\mu_2\in\mathcal{P}(\R_+),\quad t\geq0
$$
 and
$$
\mathbb{B}_t\circ\mathbb{B}_s=\mathbb{B}_{t+s} \text{~on~}\mathcal{P}(\R),\qquad s,t\geq0. 
$$

In this paper, we realize $\mathbb{B}_t$ as a free multiplicative subordination measure, giving a new look at the semigroup $\{\mathbb{B}_t \}_{t \geq 0}$ and generalizing it. 
That is, we are able to show that $\mathbb{B}_t(\mu)$ satisfies the formula
$$
 \bmsigma_{t}\boxtimes \mu=\bmsigma_t \calt\mathbb{B}_t(\mu),\qquad \mu \in \mathcal{P}(\R_+), 
$$
where $\calt$ denotes monotone multiplicative convolution and $\bmsigma_t$ is the Bernoulli law
$$
\bmsigma_t=\frac{t}{1+t}\delta_0+\frac{1}{1+t}\delta_{1+t}. 
$$
We can generalize the map $\mathbb{B}_t$ by changing $\bmsigma_t$ to any other measure $\sigma$ on $\R_+$ $(\sigma \neq \delta_0)$ or $\sigma$ on $\tor$, to obtain the map $\mathbb{B}_\sigma$ defined by 
\begin{equation}\label{subordination0}
\sigma\boxtimes\mu=\sigma\calt\mathbb{B}_\sigma(\mu),
\end{equation}
where $\sigma \in \mathcal{P}(\R_+)\setminus\{\delta_0\},\mu\in\mathcal{P}(\R_+)$ or $\sigma,\mu\in\mathcal{P}(\tor)$. 
This relation means that $\mathbb{B}_\sigma(\mu)$ is exactly the \emph{multiplicative subordination} for the convolution $\boxtimes$ (see \cite{Biane, BB07} for more details and \cite{L08} for an operator model) and then $\mathbb{B}_t$ appears in the special case when $\sigma=\bmsigma_t.$ 
One may think of this generalization artificial, but important properties of $\{\mathbb{B}_t\}_{t\geq0}$ remain true for $\mathbb{B}_\sigma$, as the following results show.  

\begin{theorem}\label{mainthm1} 
Let $\mu_1,\mu_2\in\mathcal{P}(\R_+)$, $\sigma,\sigma_1,\sigma_2 \in \mathcal{P}(\R_+)\setminus\{\delta_0\}$ or let $\mu_1,\mu_2,\sigma_1,\sigma_2,\sigma\in\mathcal{P}(\tor)$. 
Then 
\begin{equation}\label{eq main1}
\mathbb{B}_\sigma(\mu_1\boxtimes\mu_2)=\mathbb{B}_\sigma(\mu_1)\boxtimes\mathbb{B}_\sigma(\mu_2)
\end{equation}
and 
\begin{equation}\label{eq main2}
\mathbb{B}_{\sigma_1}\circ\mathbb{B}_{\sigma_2}=\mathbb{B}_{\sigma_2\calt\sigma_1}\text{~on~} \mathcal{P}(\R_+). 
\end{equation}
\end{theorem}
Unfortunately, we are not able to show (\ref{eq main1}) in the general case $\mu_2\in\mathcal{P}(\R)$, nor (\ref{eq main2}) on $\mathcal{P}(\R)$, since the existence of the multiplicative subordination $\mathbb{B}_\sigma(\mu)$ is not known for the general case $\sigma \in \mathcal{P}(\R_+)\setminus\{\delta_0\}$ and $\mu \in \mathcal{P}(\R)$. 

The measure $\bmsigma_t$ satisfies the relation 
$$
\bmsigma_t\calt \bmsigma_s=\bmsigma_{t+s},~~s,t\geq0.  
$$
Hence, in terms of Theorem \ref{mainthm1}, the property $\mathbb{B}_t\circ\mathbb{B}_s=\mathbb{B}_{t+s}$  is a consequence of the fact that $\{\bmsigma_t\}_{t\geq0}$ forms a 
$\calt$-convolution semigroup.

The transformation $\mathbb{B}_\sigma$ is versatile. For example, we may realize free and Boolean convolution powers (modulo some dilation) by choosing the right measures $\sigma$. For Boolean convolution powers we choose $\sigma=\delta_{1/t}$, $t>0$ and for free convolution powers we choose the Bernoulli law $\sigma=(1-1/t)\delta_0+(1/t)\delta_1$, $t\geq1$. 
Then the relations 
$$
\D_{1/t}(\mu\boxtimes\nu)^{\boxplus t}=\D_{1/t}(\mu)^{\boxplus t} \boxtimes \D_{1/t}(\mu)^{\boxplus t}
$$
and
$$
\D_{1/t}(\mu\boxtimes\nu)^{\uplus t}=\D_{1/t}(\mu)^{\uplus t} \boxtimes \D_{1/t}(\mu)^{\uplus t}, 
$$
which were proved in \cite{BN08}, are just special cases of (\ref{eq main1}). Note here that $\D_t$ is the dilation of a measure by $t$. 

On the other hand, the maps $\mathbb{B}_t$ can describe the laws of a free Brownian motion. Let $F_\mu$ be the reciprocal Cauchy transform and $\eta_\mu$ be the $\eta$-transform. The domains $F_\mu,\eta_\mu$ are respectively $\comp^+,\comp^-$ when $\mu\in\mathcal{P}(\R)$ and $\disc^c,\disc$ when $\mu\in\mathcal{P}(\tor)$. 
 It is shown in \cite{BN08} that the function
$h(t,z):=z-F_{\mathbb{B}_t(\mu)}(z)$ satisfies the complex Burgers equation
$$
\frac{\partial h}{\partial t}=-h(t,z)\frac{\partial h}{\partial z}. 
$$
We can extend this differential equation in terms of $\mathbb{B}_\sigma$. 
\begin{theorem}\label{mainthm2} 
Let $\mu\in\mathcal{P}(\R_+)$ and let $\{\nu_t\}_{t\geq0}$ be a weakly continuous $\calt$-convolution semigroup of probability measures on $\R_+$ or $\tor$ such that $\nu_0=\delta_1$ and let $K$ be its generator defined by 
$$
zK(z)=\frac{d}{dt}\bigg|_0\eta_{\nu_t}(z).  
$$
 Then the function 
$
H(t,z):=z-F_{\mathbb{B}_{\nu_t}}(z) 
$
satisfies 
$$
\frac{\partial H}{\partial t}=-zK\left(\frac{H}{z}\right)\frac{\partial H}{\partial z}.
$$
In terms of $\eta(t,z):=1-F_{\mathbb{B}_{\nu_t}(\mu)}(z)/z$, we have
$$
\frac{\partial \eta}{\partial t}=-K(\eta)\left(-z\frac{\partial \eta}{\partial z}+\eta\right).
$$
\end{theorem}
We have not been able to find out a noncommutative stochastic process whose marginal distributions are described by this differential equation.  Finding such an example may be an interesting question.

A similar idea works for additive free convolution. Let $\mathbb{A}_\nu(\mu)$ denote the probability measure (called the \emph{additive subordination}) characterized by  
\begin{equation}\label{subordination2}
\sigma\boxplus \mu=\sigma\rhd\mathbb{A}_\sigma(\mu),~~~\mu,\sigma\in\mathcal{P}(\R), 
\end{equation}
where $\rhd$ is monotone convolution. 
The additive subordination $\mathbb{A}_{\sigma}(\mu)$ exists for any $\mu,\sigma\in\mathcal{P}(\R)$; see \cite{BB07, Biane} for more details and \cite{L07} for an operator model. The additive subordination was originally introduced in \cite{V93}. 
\begin{theorem}\label{mainthm3} 
For $\mu_1, \mu_2,\sigma_1,\sigma_2,\sigma\in\mathcal{P}(\R)$, we have 
$$
\mathbb{A}_\sigma(\mu_1\boxplus\mu_2)=\mathbb{A}_\sigma(\mu_1)\boxplus\mathbb{A}_\sigma(\mu_2)
$$
and 
$$
\mathbb{A}_{\sigma_1}\circ\mathbb{A}_{\sigma_2}=\mathbb{A}_{\sigma_2\rhd\sigma_1}\text{~on~}\mathcal{P}(\R).  
$$
\end{theorem}

\begin{theorem}\label{mainthm4} 
Let $\mu\in\mathcal{P}(\R)$ and let $\{\nu_t\}_{t\geq0}$ be a weakly continuous $\rhd$-convolution semigroup of measures on $\R$ 
and let $J$ be its generator defined by 
$$
J(z)=\frac{d}{dt}\bigg|_0F_{\nu_t}(z). 
$$
Then the function
$
G(t,z):=z-F_{\mathbb{A}_{\nu_t}(\mu)}(z) 
$
satisfies 
$$
\frac{\partial G}{\partial t}=J(z-G)\frac{\partial G}{\partial z}.
$$
\end{theorem}

The paper is organized as follows. We show Theorem \ref{mainthm1}, \ref{mainthm2} in Section \ref{sec2} and then give examples, in particular focusing on the Belinschi-Nica semigroup and free and Boolean convolutions powers. 
We then show Theorem \ref{mainthm3}, \ref{mainthm4} and give some examples in Section \ref{sec3}.  

We give more examples of the multiplicative subordination $\mathbb{B}_\sigma$ when $\sigma$ is a beta distribution and Boolean stable law in Section \ref{sec4}. The former example turns out to be related to the Markov transform \cite{K98} and scale mixtures of Boolean stable laws developed in \cite{AHd}. 

Finally we consider additive subordination $\mathbb{A}_\sigma$ when $\sigma$ is a Cauchy distribution and study its properties in Section \ref{sec5}.

\section{Preliminaries}
\subsection{Notations}
We collect basic notations used in this paper. 
\begin{enumerate}[\rm(1)] 

\item $\mathcal{P}(I)$ is the set of Borel probability measures on a set $I$. 
We will use $I=\R, \R_+=[0,\infty), \R_-=(-\infty,0], \tor=\{z\in\comp: |z|=1\}$. 
\item For $a \in \R$, we denote by $\D_a(\mu)$ the dilation of a probability measure $\mu$, i.e.\ if a random variable $X$ follows $\mu$, then $\D_a(\mu)$ is the law of $a X$.

\item For $p \geq0$ and $\mu\in\mathcal{P}(\R_+)$, let $\mu^p$ be the push-forward of $\mu$ by the map $x\mapsto x^p$.

\item $\comp^+, \comp^-$ denote the complex upper half-plane and the lower half-plane, respectively. 
\item $\disc$ denotes the open unit disc of $\comp$ with center 0 and $\tor$ denotes the unit circle of $\comp$. 
\item For $\lambda,M>0$, $\Gamma_{\lambda,M}$ is the truncated cone $\{z\in\comp^+: |\text{Re}(z)| < \lambda\,\text{Im}(z), \text{Im}(z)>M\}$.

\item The logarithm $z\mapsto \log z$ is the principal value defined in $\comp \setminus\R_-$. 

\item For $p \in \R$, the power $z\mapsto z^p$ denotes the principal value $e^{p \log z}$ in $\comp \setminus\R_-$.  
\end{enumerate}

\subsection{Additive convolutions} 
Additive free, monotone and Boolean convolutions are binary operations on the set $\mathcal{P}(\R)$, denoted by $\boxplus, \rhd$ and $\uplus$ respectively. 
Each convolution arises as the sum of (free, monotonically or Boolean) independent random variables, and is characterized in terms of some transforms. 

Let $\mu \in\mathcal{P}(\R)$. Let $G_\mu$ be the \emph{Cauchy transform} $\int_{\R}\frac{1}{z-x}\,\mu(dx)$ and let $F_\mu,\eta_\mu$ be the \emph{reciprocal Cauchy transform} and \emph{$\eta$-transform}:  
\begin{align}
&F_\mu(z) = \frac{1}{G_\mu(z)},&z\in\comp^+,  \\
&\eta_\mu(z) =1-zF_\mu\left(\frac{1}{z}\right),&z\in\comp^-. \label{etaF} 
\end{align}
When $\mu\in\mathcal{P}(\R_+)$, we consider $G_\mu,F_\mu,\eta_\mu$ in $\comp\setminus \R_+$.  
It is known in \cite{BV93} that, for any $\lambda>0$, there exist $\lambda', M, M'>0$ such that $F_\mu$ is univalent in $\Gamma_{\lambda',M'}$ and $F_\mu(\Gamma_{\lambda',M'}) \supset \Gamma_{\lambda,M}$. Then the univalent compositional left inverse $F_\mu^{-1}$ exists in $ \Gamma_{\lambda,M}$, and
 the \emph{Voiculescu transform} of $\mu$ is defined by 
\begin{equation}\label{eq44}
\phi_\mu(z) := F_\mu^{-1}(z)-z,~~z\in\Gamma_{\lambda,M}. 
\end{equation}
The three convolutions are then characterized as follows: 
\begin{flalign}
&\cite{BV93}& &\phi_{\mu_1 \boxplus \mu_2}(z) = \phi_{\mu_1}(z) + \phi_{\mu_2}(z)& &\text{~in some $\Gamma_{\alpha,\beta}$},& \label{boxplus}\\
&\cite{M00} &&F_{\mu_1 \rhd\mu_2}(z) = F_{\mu_1} (F_{\mu_2}(z)),& &z \in \comp^+,& \label{rhd}\\
&\cite{SW97} &&\eta_{\mu_1 \uplus \mu_2}(z) = \eta_{\mu_1}(z) + \eta_{\mu_2}(z),&&z \in \comp^-.& \label{uplus}
\end{flalign}
For  any $\mu\in\mathcal{P}(\R)$, $t\geq1$ and $s\geq0$, probability measures $\mu^{\boxplus t}$ and $\mu^{\uplus s}$ exist such that 
$\phi_{\mu^{\boxplus t}}(z)=t\phi_\mu(z)$ (in the common domain) and that $\eta_{\mu^{\uplus s}}(z)=s\eta_\mu(z)$ in $\comp^-.$ The proofs can be found in \cite{BB05} and \cite{SW97} respectively.

%\subsection{Infinite divisibility} 
%The infinite divisibility with respect to a convolution $\star\in\{\boxplus, \rhd,\uplus\}$ is defined as follows: 
%\begin{definition}
%A probability measure $\mu$ on $\R$ is $\star$-infinitely divisible if, for any natural number $n$, one can find $\mu_n\in\mathcal{P}(\R)$ such that 
%$$
%\mu =\mu_n^{\star n}=\underbrace{\mu_n \star \cdots \star \mu_n.}_{\text{$n$ times}}
%$$ 
%\end{definition}
%Let us denote by $\id(\star)$ the set of $\star$-infinitely divisible measures. It is known that any probability measure is $\uplus$-infinitely divisible~\cite{SW97} since the Boolean power $\mu^{\uplus t}$ exists for any $\mu\in\mathcal{P}(\R)$ and $t>0$.  
%For the free case, particularly useful subclass of $\id(\boxplus)$ is the following. 
%\begin{definition}
%A probability measure $\mu$ is said to be in class $\iu$ if $F_\mu$ is univalent in $\comp^+$ and moreover, $F_\mu^{-1}$ has an analytic continuation from $F_\mu(\comp^+)$ to $\comp^+$ as a univalent function. 
%\end{definition}
%The class $\iu$ is a subclass of $\id(\boxplus)$ as proved in \cite{AHa}. 

\subsection{Multiplicative convolutions on $\R_+$}
Multiplicative free \cite{V87,BV93} and monotone \cite{B05,F09a} convolutions are binary operations on $\mathcal{P}(\R_+)$, denoted by $\mu\boxtimes\nu$ and $\mu\calt\nu$, respectively. 
Let $X, Y$ be positive random variables with distributions $\mu$ and $\nu$ respectively. Then  $\mu\boxtimes\nu$ is the law of $X^{1/2}YX^{1/2}$ when $X,Y$ are free, and $\mu\calt\nu$ is the law of $X^{1/2}YX^{1/2}$ when $X-1,Y-1$ are monotonically independent. We can extend the definition for $\mu \in\mathcal{P}(\R_+)$ and an arbitrary $\nu\in \mathcal{P}(\R)$; see \cite{BV93} and \cite{AHd}.  

It is always true that $\eta_\mu(-0)=0$. If moreover $\delta_0\neq \mu  \in \mathcal{P}(\R_+)$, the function $\eta_\mu$ is strictly increasing in $(-\infty, 0)$, so that 
one can define $\eta_\mu^{-1}(z)$ and the \emph{$\Sigma$-transform}
\begin{equation}\label{sigma}
\Sigma_\mu(z) :=\frac{\eta_\mu^{-1}(z)}{z}
\end{equation}
in the interval $(\eta_\mu(-\infty),0)$. 
Suppose that $\mu,\nu\in\mathcal{P}(\R_+)$ and $\mu \neq \delta_0 \neq \nu$. 
Then the multiplicative free and monotone convolutions are characterized by 
\begin{flalign}
&\cite{BV93} &&\Sigma_{\mu \boxtimes \nu}(z) = \Sigma_\mu(z) \Sigma_\nu(z),~~~~~\,\,z\in (\alpha,0),&& \label{boxtimes}\\
&\cite{B05} && \eta_{\mu \calt\nu}(z) = \eta_{\mu} (\eta_{\nu}(z)),~~~~~z \in (-\infty,0),&& \label{calt}
\end{flalign}
where $\alpha:=\max\{\eta_{\mu_1}(-\infty), \eta_{\mu_2}(-\infty) \}$. 

For $\delta_0\neq \mu\in\mathcal{P}(\R_-)$, the formula (\ref{boxtimes}) is still true if we define $\Sigma_\mu(z):=-\Sigma_{\D_{-1}(\mu)}(z)$ in $(\eta_{\D_{-1}(\mu)}(-\infty),0)$. This definition of $\Sigma_{\mu}$ is compatible with the property 
$\Sigma_{\D_s(\tau)}(z)=\frac{1}{s}\Sigma_{\tau}(z)$ for $s>0$, $\delta_0\neq\tau\in\mathcal{P}(\R_+).$  Moreover, the formula (\ref{boxtimes}) may be extended to the case where both $\mu\in\mathcal{P}(\R_+), \nu\in\mathcal{P}(\R)$ have compact supports \cite{V87,RS07} and the case where $\mu\in\mathcal{P}(\R_+)$ and $\nu$ is symmetric \cite{APA09}. The formula~(\ref{calt}) can be extended to the general case $\mu \in \mathcal{P}(\R_+),\nu\in\mathcal{P}(\R)$ \cite{AHd}.

If $\delta_0\neq\mu \in \mathcal{P}(\R_+)$, then the convolution power $\mu^{\boxtimes t} \in \mathcal{P}(\R_+)$, satisfying $\Sigma_{\mu^{\boxtimes t}}(z)=(\Sigma_\mu(z))^t$, is well defined for any $ t \geq 1$ \cite{BB05}.

\subsection{Multiplicative convolutions on $\tor$}
Multiplicative free and monotone convolutions are also defined on $\mathcal{P}(\tor)$, denoted by the same symbols $\boxtimes$ and $\calt$ respectively (see \cite{V87} and \cite{B05}). 
Let $U$ and $V$ unitaries with distributions $\mu$ and $\nu$, respectively. The measure $\mu\boxtimes\nu$ is the distribution of $UV$ when 
$U$ and $V$ are free, and $\mu\calt \nu$ is the distribution of $UV$ when $U-1$ and $V-1$ are monotonically independent.   

Let $\mu\in\mathcal{P}(\tor)$. Now we consider $G_\mu(z)$ and $F_\mu(z)$ outside the unit disc $\disc$, and $\eta_\mu(z) =1-zF_\mu\left(\frac{1}{z}\right)$ in the unit disc $\disc$.  
Let $m_1(\mu):=\int_{\tor} w\, d\mu(w) \neq0$ be the first moment of $\mu$. The function $\eta_\mu$ has a convergent series expansion $\eta_\mu(z)=m_1(\mu)z+ o(z)$, and so  
one can define $\eta_\mu^{-1}(z)$ and 
\begin{equation}\label{sigma}
\Sigma_\mu(z) :=\frac{\eta_\mu^{-1}(z)}{z}
\end{equation}
in a neighborhood of  $0$. 
Suppose that $m_1(\mu), m_1(\nu)\neq0$. Then the multiplicative free convolution is characterized by \cite{BV93}
\begin{equation}
 \Sigma_{\mu \boxtimes \nu}(z) = \Sigma_\mu(z) \Sigma_\nu(z)~~\text{in a neighborhood of 0}.\label{boxtimes2}
\end{equation}
The multiplicative monotone convolution of $\mu,\nu \in \mathcal{P}(\tor)$ is characterized by \cite{B05}
\begin{equation}
  \eta_{\mu \calt\nu}(z) = \eta_{\mu} (\eta_{\nu}(z)),~~~~~z \in \disc. \label{calt2}
\end{equation}

\subsection{Positive stable, free Poisson and beta distributions}\label{subsection stable}

Let $\alpha \in(0,1]$. We denote by $\mathbf{b}_\alpha^+, \mathbf{f}_\alpha^+,\mathbf{m}_\alpha^+$ Boolean, free and monotone \emph{positive stable distributions}, respectively. They are characterized as (\ref{booleeta}), (\ref{phifree}), (\ref{F monotone}).  Similarly, we denote by  
 $\mathbf{b}_\alpha^-, \mathbf{f}_\alpha^-, \mathbf{m}_\alpha^-$ the negative stable distributions (i.e.\ the reflection of positive ones regarding the point $0$). 

Let $\MP$ be the free Poisson distribution with density 
$$
\frac{1}{2\pi}\sqrt{\frac{4-x}{x}}1_{(0,4)}(x)\,dx. 
$$
Let $\Beta_{p,q}$ be the beta distribution with parameters $p, q>0$: 
$$
\Beta_{p,q}(dx)= \frac{1}{B(p,q)}x^{p-1}(1-x)^{q-1} 1_{(0,1)}(x)\,dx, 
$$
where $B(p,q)$ is the beta function. 
In particular we put $\Beta_\alpha:=\Beta_{\alpha,1-\alpha}$ for $\alpha \in(0,1)$. By the weak continuity, we define $\Beta_{0}=\delta_0$ and $\Beta_{1}=\delta_1$.  

Useful relations are summarized here.   
\begin{align}
&F_{\mathbf{b}_{\alpha}^+}(z)= z-(-z)^{1-\alpha}, & (\text{see~} \cite{SW97})\\ 
&\eta_{\mathbf{b}_{\alpha}^+}(z)=-(-z)^{\alpha}, \label{booleeta}\\
&\Sigma_{\mathbf{b}_{\alpha}^+}(z)=(-z)^{\frac{1-\alpha}{\alpha}}, \label{sigmaboole} \\ 
&\phi_{\mathbf{f}_\alpha^+}(z)= (-z)^{1-\alpha},&(\text{see~} \cite{BV93}) \label{phifree} \\
&\Sigma_{\MP}(z) =1-z, \\
&F_{\Beta_\alpha}(z) = z\left(1-\frac{1}{z}\right)^\alpha, & (\text{see~} \cite{Ha}) \\
&\eta_{\Beta_\alpha}(z)=1-(1-z)^\alpha,  \label{betaeta} \\
&F_{\mathbf{m}_\alpha^+}(z)=-((-z)^\alpha +1)^{1/\alpha}, & (\text{see~} \cite{H10}) \label{F monotone}\\
&\eta_{\mathbf{m}_\alpha^+}(z)= 1-((-z)^\alpha+1)^{1/\alpha}, \\
&\Sigma_{\mathbf{m}_{\alpha}^+}(z)=\frac{((1-z)^\alpha-1)^{1/\alpha}}{-z},\label{sigmamonotone}
\end{align}
for $z<0$. Moreover, for $\mu \in\mathcal{P}(\R_+)$, the following hold: 
\begin{align}
&G_{\D_{-1}(\mu)}(z)=-G_\mu(-z), ~~~z\in \comp\setminus \R_-; \\
&\eta_{\D_{-1}(\mu)}(z)=\eta_\mu(-z), ~~~z\in \comp\setminus \R_-; \\
&\Sigma_{\D_{-1}(\mu)}(z)=-\Sigma_\mu(z), ~~~z\in (\eta_\mu(\infty),0). 
\end{align}
These formulas give us the characterization of negative stable distributions.

\section{The main results: Multiplication case}\label{sec2}
\subsection{Proof of Theorem \ref{mainthm1}} 
We start from the positive real line case.  Let $\sigma \in\mathcal{P}(\R_+)\setminus\{\delta_0\}, \mu\in\mathcal{P}(\R_+)$. 
The multiplicative subordination $\mathbb{B}_\sigma(\mu)$ is characterized by 
\begin{equation}\label{subordination}
\Sigma_{\mathbb{B}_\sigma(\mu)}(z)=\Sigma_{\mu}(\eta_\sigma(z)),~~z\in(-\alpha,0)
\end{equation}
for some $\alpha>0,$ see \cite{Biane,BB07}.
The homomorphism property then follows easily from Eq.\ (\ref{subordination}): 
\begin{eqnarray*}\Sigma_{\mathbb{B}_\sigma(\mu_1\boxtimes\mu_2)}(z)&=&\Sigma_{\mu_1\boxtimes\mu_2}(\eta_\sigma(z))\\&=&
\Sigma_{\mu_1}(\eta_\sigma(z))\Sigma_{\mu_2}(\eta_\sigma(z))\\&=&\Sigma_{\mathbb{B}_\sigma(\mu_1)}(z)\Sigma_{\mathbb{B}_\sigma(\mu_2)}(z)
\\&=&\Sigma_{\mathbb{B}_\sigma(\mu_1)\boxtimes \mathbb{B}_\sigma(\mu_2)}(z) 
\end{eqnarray*}
in some interval $(-\beta,0)$ and so $\mathbb{B}_\sigma(\mu_1\boxtimes\mu_2)=\mathbb{B}_\sigma(\mu_1)\boxtimes \mathbb{B}_\sigma(\mu_2).$ The compositional relation is proved as follows: 
\begin{eqnarray*}
\Sigma_{\mathbb{B}_{\sigma_1}(\mathbb{B}_{\sigma_2}(\mu))}(z)
&=&\Sigma_{\mathbb{B}_{\sigma_2}(\mu)}(\eta_{\sigma_1}(z))=\Sigma_{\mu}(\eta_{\sigma_2}(\eta_{\sigma_1}(z))\\
&=&\Sigma_{\mu}(\eta_{\sigma_2\calt\sigma_1}(z))=\Sigma_{\mathbb{B}_{\sigma_2\calt\sigma_1}(\mu)}(z)
\end{eqnarray*}
in some interval $(-\gamma,0)$ and hence $\mathbb{B}_{\sigma_1}\circ \mathbb{B}_{\sigma_2} = \mathbb{B}_{\sigma_2\calt\sigma_1}$ on $\mathcal{P}(\R_+)$. 

On the unit circle case, the same computation is valid in suitable neighborhoods of 0 of $\Sigma$-transforms when probability measures have non-vanishing means. If probability measures have vanishing means, then we may use the results of \cite{RS07} or we may resort to an approximation argument by using moments. 

\subsection{Proof of Theorem \ref{mainthm2}}
The measures $\nu_t$ satisfy $\eta_{\nu_t}\circ \eta_{\nu_s} =\eta_{\nu_{t+s}}$ and $\eta_{\nu_0}(z)=z$. By differentiating this compositional relation as regards $s$ at $0$, we obtain the partial differential equation 
\begin{equation}\label{partial1}
\frac{\partial \eta_{\nu_t}}{\partial t}(z)= zK(z)\frac{\partial \eta_{\nu_t}}{\partial z}(z). 
\end{equation}
%The relation (\ref{subordination}) can be written as 
%\begin{equation}\label{subordination2}
%\eta_t(z):=\ta_{\mathbb{B}_{\nu_t}(\mu)}^{-1}(z) = \frac{z}{\eta_{\nu_t}(z)}\eta_\mu^{-1}(\eta_{\nu_t}(z)). 
%\end{equation}
Let $\eta=\eta(t,z):= \eta_{\mathbb{B}_{\nu_t}(\mu)}(z)$ for simplicity, and let $\eta^{-1}$ denote the compositional inverse with respect to $z$. 
From (\ref{subordination}) and (\ref{partial1}), the partial derivatives $\frac{\partial  \eta^{-1}}{\partial t}$ and $\frac{\partial  \eta^{-1}}{\partial z}$ satisfy the following: 
\begin{equation}\label{eq0001}
\frac{\frac{\partial  \eta^{-1}}{\partial t}(t,z)}{\frac{\partial  \eta^{-1}}{\partial z}(t,z)-\frac{\eta^{-1}(t,z)}{z}}=\frac{z \Sigma_\mu'(\eta_{\nu_t}) \frac{\partial  \eta_{\nu_t}}{\partial t}(t,z)}{z \Sigma_\mu'(\eta_{\nu_t}) \frac{\partial  \eta_{\nu_t}}{\partial z}(t,z)} = zK(z). 
\end{equation}
Differentiate $\eta(t, \eta^{-1}(t,z))=z$ with respect to $t$ and $z$, and then we have 
\begin{equation}\label{eq0002}
\frac{\partial  \eta^{-1}}{\partial t}(t,\eta)=-\frac{\frac{\partial  \eta}{\partial t}(t,z)}{\frac{\partial  \eta}{\partial z}(t,z)},~~~~\frac{\partial  \eta^{-1}}{\partial z}(t,\eta)=\frac{1}{\frac{\partial  \eta}{\partial z}(t,z)}. 
\end{equation}
Replace $z$ by $\eta(t,z)$ in Eq.\ (\ref{eq0001}) and then, thanks to (\ref{eq0002}), we have 
\begin{equation}\label{eq0004}
\frac{\partial  \eta}{\partial t} = - K(\eta) \left(\eta-z\frac{\partial  \eta}{\partial z} \right). 
\end{equation}
By the way, from (\ref{etaF}), the function $H$ can be written as
$$
H(t,z)=z-F_{\mathbb{B}_{\nu_t}}(z) =z \eta(t,1/z), 
$$ 
so that 
\begin{equation}\label{eq0003}
\frac{\partial  H}{\partial t}(t,1/z)=\frac{1}{z}\frac{\partial  \eta}{\partial t}(t,z),~~~\frac{\partial  H}{\partial z}(t,1/z)=\eta(t,z)-z\frac{\partial  \eta}{\partial z}(t,z). 
\end{equation}
Combining (\ref{eq0002}) and (\ref{eq0003}), the desired equation follows.

\subsection{Examples}
We present three examples mentioned in the introduction including the Belinschi-Nica semigroup.
\begin{proposition} 
Recall that the measures $\bmrho_t, \bmsigma_t$ are defined by 
$$
\bmrho_t=(1-t)\delta_0+t \delta_1,~~ \bmsigma_t=\frac{t}{1+t}\delta_0+\frac{1}{1+t}\delta_{1+t}. 
$$
Then 
\begin{eqnarray}
&\mathbb{B}_{\delta_{1/t}}(\mu)=\D_{1/t}(\mu^{\uplus t}),  &t>0, \label{booleanpower}\\
&\mathbb{B}_{\bmrho_{1/t}}(\mu)=\D_{1/t}(\mu^{\boxplus t}),  &t\geq1, \label{freepower}\\ 
&\mathbb{B}_{\bmsigma_t}(\mu)=(\mu^{\boxplus (1+t)})^{\uplus \frac{1}{1+t}},  &t\geq0, \label{BNsemigroup} \\
&\delta_{s} \calt \delta_{t} =\delta_{s t}, &s,t >0, \label{delta1}\\
&\bmrho_{s} \calt\bmrho_t =\bm\rho_{s t}, &s,t>0, \label{rho1}\\
&\bmsigma_{s} \calt\bmsigma_t =\bm\sigma_{s +t}, &s,t\geq 0.\label{sigma1}
\end{eqnarray}
\end{proposition}
\begin{proof}
We can compute the following: 
\begin{eqnarray}
&\eta_{\delta_{1/t}}(z)=\frac{z}{t}, \label{eq0011}\\
&\eta_{\bmrho_{1/t}}(z)= \frac{z}{t+(1-t)z},\label{eq0012}\\
&\eta_{\bmsigma_t}(z)= \frac{z}{1-tz}. \label{eq0013}
\end{eqnarray}
Then (\ref{delta1}) -- (\ref{sigma1}) are easy to prove. 
Moreover, the following formulas are known for $t>0$ and $\mu\in\mathcal{P}(\R_+)$ (\cite{BN08}): 
\begin{align}
&\Sigma_{\mu^{\uplus t}}(z) =\frac{1}{t}\Sigma_\mu\left(\frac{z}{t}\right), \label{eq0014}\\
&\Sigma_{\mu^{\boxplus t}}(z) =\frac{1}{t}\Sigma_\mu\left(\frac{z}{t+(1-t)z}\right), \label{eq0015}\\
&\Sigma_{\D_t(\mu)}(z) =\frac{1}{t}\Sigma_\mu(z). \label{eq0016}
\end{align}

From (\ref{subordination}), (\ref{eq0011}), (\ref{eq0014}) and (\ref{eq0016}), it follows that  
$$
\Sigma_{\mathbb{B}_{\delta_{1/t}}(\mu)}(z)=\Sigma_\mu\left(\frac{z}{t}\right) = \Sigma_{\D_{1/t}(\mu^{\uplus t})}(z). 
$$
The proof of (\ref{freepower}) is similar. 

We can easily show that $\eta_{\bmsigma_t}= \eta_{\bmrho_{1/(1+t)}}\circ \eta_{\delta_{1+t}}$, and hence $\bmsigma_t=\bmrho_{1/(1+t)}\calt\delta_{1+t}$. 
Now Theorem \ref{mainthm1} is applicable: 
$$
\mathbb{B}_{\bmsigma_t}(\mu)= \mathbb{B}_{\delta_{1+t}} (\mathbb{B}_{\bmrho_{1/(1+t)}}(\mu)) = (\mu^{\boxplus (1+t)})^{\uplus \frac{1}{1+t}}. 
$$
Note that the dilation operator $\D_s$ commutes with the free power and Boolean power. 
\end{proof}

The following "commutation relation"  was proved in \cite{BN08} and was crucial in the proof of the semigroup property $\mathbb{B}_t\circ\mathbb{B}_s=\mathbb{B}_{t+s}$. We are able to give a different proof in terms of multiplicative subordination functions. 

\begin{proposition} Let $\mu \in \mathcal{P}(\R)$ and $p, q$ be two real numbers such that $p \geq 1$ and $1-\frac{1}{p} < q$. Then
\begin{equation}
(\mu^{\boxplus p})^{\uplus q}= (\mu^{\uplus q'})^{\boxplus p'},
\end{equation}
where $p', q'$ are defined by $p' := p q/(1 - p + p q)$ and $q' := 1 - p + p q$.
\end{proposition}
\begin{proof}
This is a straight consequence of the relation 
$$
\bmrho_{1/p} \calt \delta_{1/q}=\delta_{1/q'}\calt \bmrho_{1/p'} 
$$ 
and  (\ref{eq main2}).
\end{proof}

When $\mu=\nu$, the subordination map has a special meaning. 
\begin{example} \label{mu=nu}
The map $\mu\mapsto \mathbb{B}_\mu(\mu)$ is the multiplicative Boolean-to-free Bercovici-Pata map on $\mathcal{P}(\R_+)\setminus\{\delta_0\}$, since 
$\Sigma_{\mathbb{B}_\mu(\mu)}(z)= \Sigma_\mu(\eta_\mu(z))=\frac{z}{\eta_\mu(z)}$ \cite{AHc}. The measure $\mathbb{B}_\mu(\mu)$ is $\boxtimes$-infinitely divisible for any $\mu\in\mathcal{P}(\R_+)\setminus\{\delta_0\}$, but this map is not surjective onto the set of $\boxtimes$-infinitely divisible measures.
\end{example}

\begin{example}
We specialize in $\mathbb{B}_t$ by taking $\nu_t :=\bmsigma_t$. In this case $\eta_{\nu_t}=\frac{z}{1-tz}$ and so 
$$
K(z)=\frac{1}{z}\frac{d}{dt}\bigg|_0\eta_{\nu_t}(z)=z.
$$
Thus we recover the complex Burgers equation of Belinschi and Nica: 
$$
\frac{\partial H}{\partial t}=-H\frac{\partial H}{\partial z}. 
$$
\end{example}

\begin{example} 
Let $\nu_t=\delta_{e^{-t}}$ and consider $\mathbb{B}_{\delta_{e^{-t}}}(\mu)=\D_{e^{-t}}\mu^{\uplus e^t}$. Then $\eta_{\nu_t}=ze^{-t}$ and so 
$$
K(z)=\frac{1}{z}\frac{d}{dt}\bigg|_0\eta_{\nu_t}(z)=-1, 
$$ 
and so we obtain 
$$
\frac{\partial H}{\partial t}=z\frac{\partial H}{\partial z}. 
$$
\end{example}

\begin{example} 
Since $\{\bmrho_t\}_{t\geq0}$ satisfies $\bmrho_s \calt \bmrho_t = \bmrho_{st}$ for $s,t \geq 0$, we put $\nu_t=\bmrho_{e^{-t}}$. Consider the composition semigroup $\mathbb{B}_{\nu_t}(\mu)=\D_{e^{-t}}\mu^{\boxplus e^t}$. Then 
$$
K(z)=\frac{1}{z}\frac{d}{dt}\bigg|_0\frac{z}{e^t+(1-e^t)z}=z-1.
$$  
Thus the differential equation becomes
$$
\frac{\partial H}{\partial t}=(z-H)\frac{\partial H}{\partial z}. 
$$
\end{example}

\begin{example} The Boolean stable law $\mathbf{b}_\alpha^+$ satisfies $\mathbf{b}_\alpha^+\calt \mathbf{b}_\beta^+ =\mathbf{b}_{\alpha\beta}^+$ for $\alpha,\beta\in(0,1]$; see (\ref{booleeta}). So we define $\nu_t=\mathbf{b}^+_{e^{-t}}$ to obtain a composition semigroup. The generator is computed as 
$$
K(z) = -\frac{1}{z}\frac{d}{dt}\bigg|_0 (-z)^{e^{-t}} = -\log(-z). 
$$  
The differential equation for $H(t,z)$ is: 
$$
\frac{\partial H}{\partial t}=z\log\left(\frac{H}{z}\right)\frac{\partial H}{\partial z}. 
$$
\end{example}

\begin{example} The beta distribution $\Beta_{\alpha}$  satisfies $\Beta_{\alpha} \calt \Beta_{\gamma} = \Beta_{\alpha\gamma}$ for $\alpha,\gamma \in[0,1]$ and $\Beta_1=\delta_1$; see (\ref{betaeta}). So we define $\nu_t:=\Beta_{e^{-t}}$ and consider the composition semigroup $\mathbb{B}_{\nu_t}$. The generator $K$ is given by 
$$
K(z) = \frac{1}{z}\frac{d}{dt}\bigg|_0\left(1- (1-z)^{e^{-t}}\right) = \frac{(1-z)\log(1-z)}{z},  
$$
and so the differential equation for $H(t,z)$ is given by  
$$
\frac{\partial H}{\partial t}=z\left(\frac{z}{H}-1\right)\log\left(1-\frac{H}{z}\right)\frac{\partial H}{\partial z}. 
$$
\end{example}
The maps $\mathbb{B}_{\mathbf{b}_\alpha}$ and $\mathbb{B}_{\Beta_\alpha^+}$ will be investigated more in Section \ref{sec4}.

\section{The main results: Additive case}\label{sec3}
\subsection{Proof of Theorem \ref{mainthm3}}
The measure $\mathbb{A}_\nu(\mu)$ is characterized by 
\begin{equation}\label{subordination'}
\phi_{\mathbb{A}_\nu(\mu)}(z)=\phi_{\mu}(F_\nu(z)),~~z\in \Gamma_{\lambda,M}
\end{equation}
for some $\lambda, M>0.$ 
The homomorphism property follows from (\ref{subordination'}): 
\begin{equation*}
\begin{split}
\phi_{\mathbb{A}_\nu(\mu_1\boxplus\mu_2)}(z)&=\phi_{\mu_1\boxplus\mu_2}(F_\nu(z))=\phi_{\mu_1}(F_\nu(z))+\phi_{\mu_2}(F_\nu(z))\\
&=\phi_{\mathbb{A}_\nu(\mu_1)}(z)+\phi_{\mathbb{A}_\nu(\mu_2)}(z)=\phi_{\mathbb{A}_\nu(\mu_1)\boxplus \mathbb{A}_\nu(\mu_2)}(z), 
\end{split}
\end{equation*}
and so $\mathbb{A}_\nu(\mu_1\boxplus\mu_2)=\mathbb{A}_\nu(\mu_1)\boxplus \mathbb{A}_\nu(\mu_2).$ 
The compositional relation is proved as follows: 
\begin{eqnarray*}
\phi_{\mathbb{A}_{\nu_1}(\mathbb{A}_{\nu_2}(\mu))}(z)
&=&\phi_{\mathbb{A}_{\nu_2}(\mu)}(F_{\nu_1}(z))=\phi_{\mu}(F_{\nu_2}(F_{\nu_1}(z))\\
&=&\phi_{\mu}(F_{\nu_2\rhd\nu_1}(z))=\phi_{\mathbb{A}_{\nu_2\rhd\nu_1}(\mu)}(z), 
\end{eqnarray*}
and hence $\mathbb{A}_{\nu_1}\circ \mathbb{A}_{\nu_2} = \mathbb{A}_{\nu_2\rhd\nu_1}$. 

\subsection{Proof of Theorem \ref{mainthm4}} 
The measures $\nu_t$ satisfy $F_{\nu_t}\circ F_{\nu_s} =F_{\nu_{t+s}}$ and $F_{\nu_0}(z)=z$. 
By differentiating this compositional relation as regards $s$ at $0$, the following partial differential equation is obtained: 
\begin{equation}\label{partial2}
\frac{\partial F_{\nu_t}}{\partial t}(z)= J(z)\frac{\partial F_{\nu_t}}{\partial z}(z). 
\end{equation}
Let $F=F(t,z):= F_{\mathbb{A}_{\nu_t}(\mu)}(z)$, and let $F^{-1}$ denote the compositional inverse with respect to $z$. 
From (\ref{subordination'}) and (\ref{partial2}), the partial derivatives $\frac{\partial  F^{-1}}{\partial t}$ and $\frac{\partial  F^{-1}}{\partial z}$ satisfy the following: 
\begin{equation}\label{eq00001}
\frac{\frac{\partial  F^{-1}}{\partial t}(t,z)}{\frac{\partial  F^{-1}}{\partial z}(t,z)-1}=\frac{\phi_\mu'(F_{\nu_t}) \frac{\partial  F_{\nu_t}}{\partial t}(t,z)}{\phi_\mu'(F_{\nu_t}) \frac{\partial  F_{\nu_t}}{\partial z}(t,z)} = J(z). 
\end{equation}
Differentiating $F(t, F^{-1}(t,z))=z$ with respect to $t$ and $z$,  we have 
\begin{equation}\label{eq00002}
\frac{\partial  F^{-1}}{\partial t}(t,F)=-\frac{\frac{\partial  F}{\partial t}(t,z)}{\frac{\partial  F}{\partial z}(t,z)},~~~~\frac{\partial  F^{-1}}{\partial z}(t,F)=\frac{1}{\frac{\partial  F}{\partial z}(t,z)}. 
\end{equation}
Replace $z$ by $F(z)$ in Eq.\ (\ref{eq00001}), thanks to (\ref{eq00002}), we have 
\begin{equation}\label{eq00004}
\frac{\partial  F}{\partial t} = J(F) \left(\frac{\partial F}{\partial z}-1 \right). 
\end{equation}
The conclusion therefore follows. 

\subsection{Examples}

The following is the additive analog of Example \ref{mu=nu}. 
\begin{example}
The map $\mu\mapsto \mathbb{A}_\mu(\mu)$ is the additive Boolean-to-free Bercovici-Pata bijection on $\mathcal{P}(\R)$, since 
$\phi_{\mathbb{A}_\mu(\mu)}(z)= \phi_\mu(F_\mu(z))=z-F_\mu(z)$. A measure $\nu$ can  be written as  $\nu=\mathbb{A}_\mu(\mu)$ for some $\mu\in\mathcal{P}(\R)$ if and only if  $\nu$ is $\boxplus$-infinitely divisible.
\end{example}

\begin{example}  
Let $\nu_t$ be the arcsine law with mean 0 and variance $t$: 
$$
\nu_t(dx)= \frac{1}{\pi\sqrt{2t-x^2}}1_{(-\sqrt{2t}, \sqrt{2t})}(x)\,dx. 
$$
Then it holds that $\nu_{s+t}=\nu_s \rhd \nu_t$, $s,t\geq0$, $\nu_0=\delta_0$, because $F_{\nu_t}$ is given by $\sqrt{z^2-2t}$. The generator is given by $J(z)=-\frac{1}{z}$. 
So the differential equation of Theorem \ref{mainthm4} becomes
$$
\frac{\partial G}{\partial t}=\frac{1}{G-z}\frac{\partial G}{\partial z}.
$$
 \end{example}

\begin{example}  
Let $\nu_t$ be $\mathbf{c}_{at,bt}$, where $\mathbf{c}_{a,b}$ is the Cauchy distribution 
$$
\mathbf{c}_{a,b}(dx) = \frac{b}{\pi[(x-a)^2 + b^2]}1_{\R}(x)\,dx,~~a \in \R, ~~b > 0. 
$$ 
We define $\mathbf{c}_{a,0}=\delta_a$ considering the weak continuity. 
Then $\nu_{s+t}=\nu_s \rhd \nu_t$, $s,t\geq0$, $\nu_0=\delta_0$. The generator $J$ is the constant $-a+bi$, and so the differential equation is 
$$
\frac{\partial G}{\partial t}=(-a+bi)\frac{\partial G}{\partial z}.
$$
 \end{example}

\section{Examples of multiplicative free subordination}\label{sec4}
We study the subordination map $\mathbb{B}_\nu$ when $\nu=\Beta_\alpha$ or $\mathbf{b}_\alpha^+$ and find its relation with the Markov transform. Many computations in this section may be extended to symmetric probability measures by using the $S$-transform for symmetric measures, but we restrict ourselves to the positive and negative cases. 

\subsection{Transforms $\MM_\alpha^+$ and $\U_\alpha^+$}
\begin{definition}
\begin{enumerate}[\rm(1)]
\item For $\mu\in\mathcal{P}(\R_+)$ and $\alpha \in (0,1]$, let $\MM_\alpha^+: \mathcal{P}(\R_+)\to\mathcal{P}(\R_+)$ be the map defined by 
\begin{equation}\label{eq04} 
\MM_\alpha^+(\mu) =  \B_{\Beta_\alpha}(\mu)^{\boxtimes 1/\alpha} \boxtimes \mathbf{m}_\alpha^+. \end{equation}
\item For $\mu\in\mathcal{P}(\R_+)$ and $\alpha \in (0,1]$, let $\U_\alpha^+: \mathcal{P}(\R_+)\to\mathcal{P}(\R_+)$ be the map defined by 
\begin{equation}\label{eq0411}
\U_\alpha^+(\mu)=\B_{\mathbf{b}_\alpha^+}(\mu)^{\boxtimes 1/\alpha}.  
\end{equation} 
\end{enumerate}
\end{definition}

\begin{theorem}\label{theorem21}
\begin{enumerate}[\rm(1)]
\item 
The probability measure $\MM_\alpha^+(\mu)$ is characterized by  
\begin{equation}
F_{\MM_\alpha^+(\mu)}(z)= -(-F_\mu(-(-z)^\alpha))^{1/\alpha},\qquad z<0, \quad \mu\in\mathcal{P}(\R_+).  \label{eq12}
\end{equation}
The map $\MM_\alpha^+$ satisfies a semigroup and a homomorphism properties: 
\begin{align}
&\MM_\alpha^+ \circ \MM_\alpha^+ =\MM_{\alpha\beta}^+;  \label{eq041} \\
&\MM_\alpha^+(\mu\rhd\nu)= \MM_\alpha^+(\mu)\rhd \MM_\alpha^+(\nu). \label{eq0410} 
\end{align}

\item The probability measure $\U_\alpha^+(\mu)$ is characterized by  
\begin{equation}
\eta_{\U_\alpha^+(\mu)}(z) = -\left(-\eta_{\mu}(-(-z)^\alpha)\right)^{1/\alpha},\qquad z<0,\quad\mu\in\mathcal{P}(\R_+).  \label{eq0041}  
\end{equation}
Moreover,  we have a semigroup and  homomorphism properties: 
\begin{align}
&\U_\alpha^+ \circ \U_\beta^+ =\U_{\alpha\beta}^+; \label{U1} \\
&\U_\alpha^+(\mu\calt\nu)= \U_\alpha^+(\mu)\calt \U_\alpha^+(\nu); \label{U2}\\
&\U_\alpha^+(\mu \boxtimes \nu)=\U_\alpha^+(\mu)\boxtimes \U_\alpha^+(\nu). \label{U3}  
\end{align}

\end{enumerate}
\end{theorem}
\begin{remark}\label{rem7}
\begin{enumerate}[\rm(1)]  
\item The fact that $\MM_\alpha^+$ is a homomorphism regarding $\rhd$ may be seen as the monotonic counterpart of \cite[Theorem 4.14]{AHd}, where the maps 
\begin{align*}
&\mathbf{B}_{\alpha}^+(\mu) = \mu^{1/\alpha}\circledast\mathbf{b}_{\alpha}^+ =  \mu^{\boxtimes 1/\alpha}\boxtimes \mathbf{b}_{\alpha}^+,  \\ 
&\mathbf{N}_{\alpha}^+(\mu) = \mu^{1/\alpha}\circledast \mathbf{n}_{\alpha}^+,  \\
&\mathbf{F}_{\alpha}^+(\mu) = \mu^{\boxtimes 1/\alpha}\boxtimes \mathbf{f}_{\alpha}^+
\end{align*}
are shown to be homomorphisms with respect to $\uplus, \ast, \boxplus$, respectively (the measure $\mathbf{n}_\alpha^+$ is a classical positive $\alpha$-stable law). They also become compositional semigroups: $\mathbf{B}_{\alpha}^+ \circ \mathbf{B}_{\beta}^+=\mathbf{B}_{\alpha\beta}^+$, $\mathbf{N}_{\alpha}^+ \circ \mathbf{N}_{\beta}^+=\mathbf{N}_{\alpha\beta}^+$, $\mathbf{F}_{\alpha}^+ \circ \mathbf{F}_{\beta}^+=\mathbf{F}_{\alpha\beta}^+$. 

\item The formula (\ref{eq12}) becomes simpler on $\R_-$. We introduce the conjugated map $\MM_\alpha^-:= \D_{-1}\circ \MM_\alpha^+ \circ \D_{-1}$ on $\mathcal{P}(\R_-)$. Then 
the formula (\ref{eq04}) changes to  
\begin{align}
&\MM_\alpha^-(\mu) = \B_{\Beta_\alpha}(\D_{-1}(\mu))^{\boxtimes 1/\alpha}\boxtimes \mathbf{m}_\alpha^- . \label{minusT}
\end{align}
and (\ref{eq12}) can be written as  
\begin{equation}\label{eq602}
F_{\MM_\alpha^-(\mu)}(z)= (F_\mu(z^\alpha))^{1/\alpha},\qquad z>0. 
\end{equation}
We can also define $\U^-_\alpha$ on $\mathcal{P}(\R_-)$, but this does not simplify (\ref{eq0041}) so much. 
\end{enumerate}
\end{remark}

\begin{proof}
 We are going to show (\ref{eq602}) instead of (\ref{eq12}). Let $F(z):=(F_\mu(z^\alpha))^{1/\alpha}$ and $\eta(z):= 1-zF(1/z)$. Using the relation (\ref{etaF}) twice, we have
$$
\eta(z) = 1-z\left(F_\mu\left(\frac{1}{z^\alpha}\right)\right)^{1/\alpha} = 1-(1-\eta_\mu(z^\alpha))^{1/\alpha},~z>0, 
$$
so that $\eta$ is strictly increasing on $(-\infty,0)$ and it has the compositional inverse
$$
\eta^{-1}(w) = [\eta_\mu^{-1}(1-(1-w)^\alpha)]^{1/\alpha}=\left(\eta_\mu^{-1}(\eta_{\Beta_\alpha}(w)) \right)^{1/\alpha}, ~w\in(-\varepsilon,0) 
$$
for some $\varepsilon>0$. Hence 
\[
\begin{split}
\Sigma(w)
&:=\frac{\eta^{-1}(w)}{w} =-\left(\frac{\eta_\mu^{-1}(\eta_{\Beta_\alpha}(w))}{(-w)^\alpha} \right)^{1/\alpha}\\
&= -\left(\frac{(-\eta_{\Beta_\alpha}(w))(- \Sigma_\mu\circ\eta_{\Beta_\alpha}(w))}{(-w)^\alpha} \right)^{1/\alpha}\\
&= \frac{((1-w)^\alpha-1)^{1/\alpha}}{-w} (\Sigma_{\mathbb{B}_{\Beta_\alpha}(\D_{-1}(\mu)}(w))^{1/\alpha}\\
&= \Sigma_{\mathbf{m}_\alpha^-}(w) \Sigma_{\mathbb{B}_{\Beta_\alpha}(\D_{-1}(\mu))^{\boxtimes 1/\alpha}}(w), \\
&= \Sigma_{\MM_\alpha^-(\mu)},\qquad w\in(-\varepsilon',0)
\end{split}
\]
for some $\varepsilon'>0$, where the formula (\ref{sigmamonotone}) was used on the fourth line. 
This computation shows (\ref{eq602}). 
(\ref{eq041}) and (\ref{eq0410}) are simple applications of (\ref{eq12}).  

The proof of (\ref{eq0041}) is following. We obtain 
\[
\begin{split}
 \frac{1}{w}\eta_{\U_\alpha^+(\mu)}^{-1}(w)
 &=\Sigma_{\U_\alpha^+(\mu)}(w)  
 =\Sigma_{\B_{\mathbf{b}_\alpha^+}(\mu)^{\boxtimes 1/\alpha}}(w) 
= \Sigma_{\B_{\mathbf{b}_\alpha^+}(\mu)}(w)^{1/\alpha} \\
&= \left(\Sigma_{\mu}(-(-w)^\alpha)\right)^{1/\alpha} = \frac{1}{-w}(-\eta_\mu^{-1}(-(-w)^{\alpha}))^{1/\alpha},~~w\in(-\delta,0) 
\end{split}
\]
for some $\delta>0$, where (\ref{subordination}) was used on the fourth equality. After some more computation we have (\ref{eq0041}). 
(\ref{U1}) and (\ref{U2}) are easy applications of (\ref{eq0041}). (\ref{U3}) follows from the relation  $\Sigma_{\U_\alpha^+(\mu)}(w)=(\Sigma_\mu(-(-w)^\alpha))^{1/\alpha}.$
\end{proof}

\begin{example}\label{ex43} What is shown below is several computations involving $\MM_\alpha^+$, $\U_\alpha^+$, $\mathbb{B}_{\Beta_\alpha}$ and  $\mathbb{B}_{\mathbf{b}_\alpha^+}$ $(0<\alpha\leq1)$.  

\begin{enumerate}[\rm(1)]
\item $\MM_\alpha^+(\mathbf{m}_\beta^+)=\mathbf{m}_{\alpha\beta}^+$. In particular, $\MM_\alpha^+(\delta_1)=\mathbf{m}_\alpha^+.$
\item\label{betaboole} 
$\MM_\alpha^+(\Beta_\alpha) =\mathbf{b}_\alpha^+$.

\item\label{bernoullibeta} $\MM_\alpha^+(\bmrho_\alpha)=\Beta_\alpha.$

\item\label{mp} $\B_{\Beta_\alpha}(\MP^{\boxtimes 1/\alpha})=\MP$, since 
$$
\Sigma_{\B_{\Beta_\alpha}(\MP^{\boxtimes 1/\alpha})}(z)= \Sigma_{\MP^{\boxtimes 1/\alpha}}( \eta_{\Beta_\alpha}(z)) = 
\Sigma_{\MP}( \eta_{\Beta_\alpha}(z))^{1/\alpha} = 1-z = \Sigma_{\MP}(z). 
$$

\item $\MM_\alpha^+(\MP) =  \mathbf{m}_\alpha^+\boxtimes \MP$ because of (\ref{mp}) above and (\ref{eq04}). 

\item In \cite{AHa}, a family of  measures is introduced. Part of that family is $\mu_{\alpha,p}$ which is characterized by 
\begin{equation} \label{eq46}
G_{\mu_{\alpha, p}}(z)=
-\left(\frac{\left(1+\left(-\frac{1}{z}\right)^{\alpha}\right)^{p}-1}{p}\right)^{1/\alpha},  ~~z\in\comp^+,  
\end{equation}
for $\alpha, p \in (0,1]$.\footnote{The parametrization is different from that in the original paper \cite{AHa}, and the dilation is omitted here. } The powers 
$z\mapsto z^{\alpha}$ and $z\mapsto z^{1/\alpha}$ are the principal value. 
We have the identity 
\begin{equation}
\mu_{\alpha,p}=\MM_\alpha^+(\Beta_{1-p,1+p}),~~~\alpha,p \in (0,1],  
\end{equation}
which can be checked by using the fact that the beta distribution $\Beta_{1-p,1+p}$ coincides with $\mu_{1,p}$ \cite{AHa}.

\item $\B_{\mathbf{b}_\alpha^+}\left(\mathbf{b}_{\frac{\alpha\beta}{1-\beta+\alpha\beta}}^+\right)=\mathbf{b}_\beta^+$, $\alpha,\beta\in(0,1]$.

\item $\MM_\alpha^+(\mu)=\mathbf{m}_\alpha^+ \calt \U_\alpha^+(\mu)$ for $\mu\in\mathcal{P}(\R_+)$.  This is proved as follows. First we can show that 
\begin{equation}\label{eq0047}
\mathbf{b}_\alpha^+ = \Beta_\alpha \calt \mathbf{m}_\alpha^+, ~~\alpha \in (0,1]. 
\end{equation}
 by computing $\eta$-transforms. Starting from the relations (\ref{eq04}) and (\ref{subordination0}), we have  
$$
\MM_\alpha^+(\mu)=  \mathbf{m}_\alpha^+ \boxtimes \B_{\Beta_\alpha}(\mu^{\boxtimes 1/\alpha}) = \mathbf{m}_\alpha^+ \calt \mathbb{B}_{\mathbf{m}_\alpha^+}(\B_{\Beta_\alpha}(\mu^{\boxtimes 1/\alpha})). 
$$
Thanks to Theorem \ref{mainthm1} and (\ref{eq0047}), we  get
$\mathbb{B}_{\mathbf{m}_\alpha^+}\circ \B_{\Beta_\alpha} = \mathbb{B}_{\Beta_\alpha \calt\mathbf{m}_\alpha^+ } = \mathbb{B}_{\mathbf{b}_\alpha^+}$.

\item $ \mu^{1/\alpha}\circledast\mathbf{b}_\alpha^+ = \mu^{\boxtimes 1/\alpha} \boxtimes\mathbf{b}_\alpha^+=\mu \calt \mathbf{b}_\alpha^+ = \Beta_\alpha \calt \MM_\alpha^+(\mu)$ for $\mu\in\mathcal{P}(\R_+)$. The first and second equalities were proved in \cite{AHd}. The last one is proved as follows: 
\[
\begin{split}
\eta_{\Beta_\alpha \calt \MM_\alpha^+(\mu)}(z) 
&=  1-(1- \eta_{\MM_\alpha^+(\mu)}(z))^\alpha =  1-\left(z F_{\MM_\alpha^+(\mu)}(1/z)\right)^\alpha\\
&= 1+(-z)^\alpha F_\mu(-(-z)^{-\alpha}) = \eta_\mu(-(-z)^\alpha),~~z<0. 
\end{split}
\]
The last expression is equal to $\eta_{\mathbf{b}_\alpha^+ \circledast \mu^{1/\alpha}}(z)$ as shown in \cite{AHd}. 
\end{enumerate}

\end{example}

\subsection{Markov transform, mixture of Boolean stable laws and $\MM_\alpha^+$}
\begin{definition}
Let $\M(\nu)$ be the Markov (or Markov-Krein) transform of $\nu$ defined by 
$$
G_{\M(\nu)}(z)=\exp\left(\int_{\R}\log\left(\frac{1}{z-x}\right)\,\nu(dx) \right); 
$$
see \cite{K98}. 
In general $\nu$ can be taken to be a Rayleigh measure, but now we consider a probability measure $\nu$ satisfying $\int_{\R}\log(2+|x|)\,\nu(dx)<\infty$. 
It is characterized by the differential equation 
\begin{equation}\label{eq Markov}
\frac{d}{dz}G_{\M(\nu)}(z)=-G_\nu(z)G_{\M(\nu)}(z).
\end{equation}
\end{definition}
We compute the Markov transforms of classical scale mixtures of Boolean stable laws $\mu^{1/\alpha}\circledast\mathbf{b}_\alpha^+ $, $\mu\in\mathcal{P}(\R_+)$.  

\begin{proposition}\label{prop0015} Let $\alpha\in(0,1]$ as usual. Then 
\begin{equation}\label{eq05}
\M( \nu^{1/\alpha}\circledast\mathbf{b}_\alpha^+) = \MM_\alpha^+(\M(\nu)),~~~\nu \in \mathcal{P}(\R_+). 
\end{equation}
\end{proposition}
\begin{remark}\begin{enumerate}[\rm(1)]
\item In terms of $\R_-$, we have   
\begin{equation}\label{eq06}
\M((\D_{-1}(\nu))^{1/\alpha}\circledast \mathbf{b}_\alpha^-) = \MM_\alpha^-(\M(\nu)),~~~\nu \in \mathcal{P}(\R_-). 
\end{equation}
\item In terms of the map $\mathbf{B}_\alpha^+$ (see Remark \ref{rem7}), we may write (\ref{eq05}) as
\begin{equation}
\M\circ \mathbf{B}_\alpha^+= \MM_\alpha^+\circ \M, 
\end{equation}
i.e.\ the Markov transform intertwines the compositional semigroups $(\mathbf{B}_\alpha^+)_{\alpha\in(0,1]}$ and $(\MM_\alpha^+)_{\alpha\in(0,1]}$. 
\end{enumerate}
\end{remark}
\begin{proof} We consider the formula (\ref{eq06}). Let $\mu=\M(\nu)$ for $\nu\in \mathcal{P}(\R_-)$. It follows from elementary calculus that 
\[
\begin{split}
\frac{d}{dz} \log G_{\MM_\alpha^-(\mu)}(z) 
&=\frac{1}{\alpha}\frac{d}{dz} \log \left(G_\mu(z^\alpha)\right)=z^{\alpha-1}\frac{G_{\mu}'(z^\alpha)}{G_{\mu}(z^\alpha)} \\
&= -z^{\alpha-1}G_{\nu}(z^\alpha)=z^{\alpha-1}G_{\D_{-1}(\nu)}(-z^\alpha)\\
&=-G_{(\D_{-1}(\nu))^{1/\alpha}\circledast \mathbf{b}_\alpha^-}(z),~~z>0.
\end{split}
\]
The last line is a consequence of the formula (4.4) in \cite{AHd}. This shows (\ref{eq06}). 
\end{proof}

\begin{example}
\begin{enumerate}[\rm(1)]
\item\label{boolemono} $\M(\mathbf{b}_\alpha^+)=\MM_\alpha^+(\delta_1)=\mathbf{m}_\alpha^+$ since $\M(\delta_1)=\delta_1.$
\item $\M((\bmrho_\alpha)^{1/\alpha}\circledast\mathbf{b}_\alpha^+ ) = \mathbf{b}_\alpha^+$, since $\MM_\alpha^+(\M(\bmrho_\alpha)) =\MM_\alpha^+(\Beta_\alpha)=\mathbf{b}_\alpha^+$ from Example \ref{ex43}(\ref{betaboole}), (\ref{bernoullibeta}).

\item $\M((\mathbf{f}_\alpha^+)^{\uplus\alpha})=\mathbf{f}_\alpha^+$. This is not directly related to Proposition \ref{prop0015}, but seems interesting to compare with $\M(\mathbf{b}_\alpha^+)=\mathbf{m}_\alpha^+$. The proof is as follows. Note that (\ref{eq Markov}) is equivalent to 
$(\log F_{\M(\nu)})'=1/F_\nu$. Recalling the relation 
\begin{equation}\label{eq free stable}
F_{\mathbf{f}_\alpha^+}^{-1}(z)= z+(-z)^{1-\alpha}, \qquad z \in \Gamma_{\lambda, M}
\end{equation}
for some $\lambda,M>0$, we get 
\begin{equation}\label{eq75}
(\log F_{\mathbf{f}_\alpha^+})' = \frac{1}{F_{\mathbf{f}_\alpha^+}+(1-\alpha)(-F_{\mathbf{f}_\alpha^+})^{1-\alpha}}. 
\end{equation}
From (\ref{eq free stable}), we have $(-F_{\mathbf{f}_\alpha^+})^{1-\alpha}=z-F_{\mathbf{f}_\alpha^+}$. Combining this with (\ref{eq75}), we obtain 
$(\log F_{\mathbf{f}_\alpha^+})'= 1/((1-\alpha)z+\alpha F_{\mathbf{f}_\alpha^+})= 1/ F_{(\mathbf{f}_\alpha^+)^{\uplus \alpha}}$, the conclusion. 

\item  We can show that the inverse Markov transform $\nu_p:=\M^{-1}(\Beta_{1-p,1+p})$ has the Cauchy transform
$$
G_{\nu_p}(z)= \frac{p}{z}\cdot\frac{(z-1)^{p-1}}{z^p-(z-1)^p} = p\left(\frac{1}{z} +\frac{(z-1)^{p-1} -z^{p-1}}{z^p - (z-1)^p}\right),  
$$
and so 
$$
\nu_p=p\delta_0+ \frac{p\sin\pi p}{\pi}\frac{x^{p-1}(1-x)^{p-1}}{x^{2p}-2x^p(1-x)^p \cos \pi p + (1-x)^{2p}}1_{(0,1)}(x)\,dx. 
$$
Using this measure, we have $\mu_{\alpha, p}=\M((\nu_p)^{1/\alpha}\circledast\mathbf{b}_\alpha^+)$ for $0<\alpha,p \leq 1$. 
Notice the relation $\nu_p=\bmrho_{1-p}\circledast \tau_p$ exists, where
$$
\tau_p(dx):= \frac{p\sin\pi p}{(1-p)\pi}\frac{x^{p-1}(1-x)^{p-1}}{x^{2p}-2x^p(1-x)^p \cos \pi p + (1-x)^{2p}}1_{(0,1)}(x)\,dx
$$
is the law of a random variable $\mathbb{G}_p$ studied in \cite{BFRY06}. 
\end{enumerate}
\end{example}

\section{Additive subordination by Cauchy distribution}\label{sec5}

We are going to study $\mathbb{A}_\nu$ in details  when $\nu$ is the Cauchy distribution $\mathbf{c}_{a,b}$. 
The corresponding map $\mathbb{A}_{\mathbf{c}_{a,b}}$ gives us freely infinitely divisible distributions with rational function densities.  The map $\mathbb{A}_{\mathbf{c}_{a,b}}$ is denoted by $\mathbb{A}_{a,b}$ below.

\begin{proposition}
For $a \in \R$, $b \geq 0$ and $\mu\in\mathcal{P}(\R)$,  $\mathbb{A}_{a,b}(\mu)$ is characterized by 
\begin{equation}\label{eq0121}
F_{\mathbb{A}_{a,b}(\mu)}(z) = F_\mu(z-a +ib) +a -ib,~~z\in\comp^+. 
\end{equation}
\end{proposition}
\begin{proof}
The Eq.\ (\ref{subordination'}) is equivalent to $F_{\mathbb{A}_{a,b}(\mu)}^{-1}(z)-a+ib=F^{-1}_\mu(z-a+ib)$ in a domain $\Gamma_{\lambda,M}$, 
and so the conclusion follows in a domain of $\comp^+$, which in fact is valid in $\comp^+$ by analyticity. 
\end{proof}
We denote by $m_1(\mu)$, $m_2(\mu)$ and $\sigma^2(\mu)$ the first moment, second moment and variance of a probability measure $\mu$, respectively. The notation $f \sim g,~|x| \to \infty$ means that $\frac{f(x)}{g(x)} \to 1$ as $|x| \to \infty$. 

\begin{proposition} 
\begin{enumerate}[\rm(1)]
\item For any $\mu,\nu\in\mathcal{P}(\R)$ and any $a \in \R,b \geq0$, we have the homomorphism properties 
\begin{align}
&\mathbb{A}_{a,b}(\mu \rhd \nu) = \mathbb{A}_{a,b}(\mu) \rhd \mathbb{A}_{a,b}(\nu), \\
&\mathbb{A}_{a,b}(\mu \boxplus \nu) = \mathbb{A}_{a,b}(\mu) \boxplus \mathbb{A}_{a,b}(\nu), \\
&\mathbb{A}_{a,b}(\mu \uplus \nu) = \mathbb{A}_{a,b}(\mu) \uplus \mathbb{A}_{a,b}(\nu).  
\end{align}
\item For any $a\in\R$, $b > 0$ and any $\mu$ which is not a point measure, $\mathbb{A}_{a,b}(\mu)$ has a real analytic density on $\R$. Moreover,  if $\mu$ has a finite variance, then the density of $\mathbb{A}_{a,b}(\mu)$  $\sim \frac{b\sigma^2(\mu)}{\pi x^4}$ as $|x| \to \infty$.  
\item If $\mu$ has a finite variance, then $m_1(\mathbb{A}_{a,b}(\mu))=m_1(\mu)$ and $m_2(\mathbb{A}_{a,b}(\mu)) = m_2(\mu)$. 
\end{enumerate}
\end{proposition}
\begin{proof}
(1)\,\,\, The second equality was proved in Theorem \ref{mainthm3}. The others can be shown from the formula (\ref{eq0121}). 

(2)\,\,\, If $\mu$ is not a point measure, $\text{Im}(F_{\mu}(z)) > \text{Im}(z)$ for $z \in \comp^+$. Using the Stieltjes inversion formula, we have 
\[
\mathbb{A}_{a,b}(\mu)(dx) = \frac{\text{Im}(F_\mu(x-a+ib)) -b}{\pi|F_\mu(x-a+ib)+a -ib|^2}dx.   
\]
The density function is real analytic and strictly positive on $\R$.  
The reciprocal Cauchy transform $F_\mu$ has the Nevanlinna representation as $F_{\mu}(z) = z - m_1(\mu) + \int_{\R}\frac{\rho_\mu(du)}{u-z}$ for a finite non-negative measure $\rho_\mu$ which satisfies $\rho_\mu(\R) = \sigma^2(\mu)$~\cite{M92}. So $F_\mu(x - a +ib) +a - ib  \sim x$ as $|x| \to \infty$ and 
\[
\text{Im}\!\left(F_\mu(x - a +ib)  +a - ib\right) = \int_{\R}\frac{b}{(u-x+a)^2+b^2}\rho_\mu(du) \sim \frac{b\sigma^2}{x^2}. 
\]
Therefore, the density of $\mathbb{A}_{a,b}(\mu)$ behaves as $\sim \frac{b\sigma^2}{\pi x^4}$ as $|x|\to \infty$. 

(3)\,\,\, The first moment $m_1(\mu)$ is characterized by $F_\mu(iy)-iy +m_1(\mu)=o(1)$ ($y\to\infty$), and so $m_1(\mathbb{A}_{a,b}(\mu))=m_1(\mu)$. 
The variance $\sigma^2(\mu)=\rho_\mu(\R)$ is equal to $\lim_{y \to \infty}|F_\mu(iy)-iy +m_1(\mu)|y$~\cite{M92}. Applying this formula to $\mathbb{A}_{a,b}(\mu)$, we have $\sigma^2(\mu) = \sigma^2(\mathbb{A}_{a,b}(\mu))$ and hence $m_2(\mathbb{A}_{a,b}(\mu)) = m_2(\mu)$. 
\end{proof}

We prove that $\mathbb{A}_{a,b}(\mu)$ is freely infinitely divisible for large $b$ if $\mu$ has finite variance. We introduce a subclass of freely infinitely divisible distributions characterized by \emph{univalent inverse} Cauchy transforms. 
\begin{definition}
A probability measure $\mu$ on $\R$ is said to be in class $\iu$ if $F_\mu^{-1}$, originally defined in some $\Gamma_{\lambda,M}$, extends to a univalent function in $\comp^+$. 
\end{definition}
The class $\iu$ is a subset of all freely infinitely divisible distributions as proved in \cite{AHa}.

\begin{theorem}\label{thm101}
Let $\mu$ be a probability measure with finite variance $\sigma^2$. Then $\mathbb{A}_{a,b}(\mu) \in \iu$ for $b \geq 2\sigma$. 
\end{theorem}
\begin{remark} Assume that $\mu$ is not $\boxplus$-infinitely divisible. 
If we use the free divisibility indicator $\phi(\mu)$, the assumption $b \geq 2\sigma$ can be weakened to $b \geq 2\sigma\sqrt{1-\phi(\mu)}$.  
\end{remark}
\begin{proof}
Let $\comp_s:=\{z \in \comp: \text{Im}(z) > s\}$. Maassen proved that $F_\mu$ is univalent in $\comp_{\sigma}$ and $F_\mu(\comp_\sigma) \supset \comp_{2\sigma}$ (\cite{M92}, Lemma 2.4). Therefore, $F_{\mathbb{A}_{a,b}(\mu)}(\comp_{-b + \sigma}) = F_\mu(\comp_{\sigma}) +a - ib \supset \comp_{2\sigma -b} \supset \comp^+$, so that $F_{\mathbb{A}_{a,b}(\mu)}^{-1}$ can be defined in $\comp^+$ as a univalent function. 
\end{proof}

Starting from atomic measures, we obtain many freely infinitely divisible measures whose densities are rational functions. 
\begin{example}\label{exa1}

Let $b > 0, n > 1, \lambda_j > 0, \sum_{j=1}^n \lambda_j =1$ and $a_j \in \R$. The measure 
$\mathbb{A}_{a,b}(\sum_{j=1}^n \lambda_j \delta_{a_j})$ has a rational function density. If $b >0$ is large enough, it is freely infinitely divisible. 

Let $\mu_p:=\frac{p}{2}(\delta_{-1} + \delta_{1}) + (1-p)\delta_0$, $0 \leq p \leq 1$. Then $G_{\mu_p}(z) = \frac{z^2 -1 +p}{z(z^2-1)}$, so that $G_{\mathbb{A}_{0,b}(\mu_p)}(z) = \frac{z^2 -b^2-1+p+2ibz}{z^3 - (1+b^2)z  +2ib(z^2-p/2)}$. 
The measure $\mathbb{A}_{0,b}(\mu_p)$ is given by 
\[
\mathbb{A}_{0,b}(\mu_p)(dx) = \frac{bp(x^2+b^2+1-p)}{\pi[x^6 + 2(b^2-1)x^4+(b^4+2b^2(1-2p) +1)x^2 +p^2b^2]}dx. 
\]
From Theorem \ref{thm101}, $\mathbb{A}_{0,b}(\mu_p) \in \iu$ for $b \geq 2\sqrt{p}$. 

In the particular case $p=1$ and $b=2$, $\mathbb{A}_{0,2}(\mu_1)$ is a scaled $t$-distribution with 3 degrees of freedom: 
\[
\mathbb{A}_{0,2}(\mu_1)(dx) = \frac{2}{\pi(1+x^2)^2}dx, ~~x \in \R.   
\]
We can see $\mathbb{A}_{0,b}(\mu_1)$ is not freely infinitely divisible for $b < 2$ as follows. The Voiculescu transform of $\mu_1$ is $\phi_{\mu_1}(z) = \frac{-z + \sqrt{z^2 + 4}}{2}$. This function cannot be analytic in $\comp_{b}$. From \cite[Theorem 5.10]{BV93}, $\mathbb{A}_{0,b}(\mu_1)$ is not freely infinitely divisible.  Note that the Student distribution is freely infinitely divisible for more parameters \cite{Ha}. 
\end{example}

\begin{remark}[Multiplicative analogue on $\tor$]
For $a \geq 0$ and $b \in \R$, let $\hat{\mathbf{c}}=\hat{\mathbf{c}}_{a,b}$ be a probability measure on $\tor$ defined by
\[
\hat{\mathbf{c}} (d\theta) = \frac{1}{2\pi}\frac{1 - e^{-2a}}{1 + e^{-2a}-2e^{-a}\cos(\theta - b)}d\theta, ~~0 \leq \theta < 2\pi.
\]
This is an analogue of the Cauchy distribution on $\R$ since both are the Poisson kernels. Let $c:=e^{-a+ib}$, then $\eta_{\hat{\mathbf{c}}}(z) = cz$ and
\[
\Sigma_{\hat{\mathbf{c}}}(z) = k_{\hat{\mathbf{c}}}(z) = c^{-1}= e^{-ib}\exp\left(a\int_{\tor}\frac{1+\zeta z}{1-\zeta z}\omega(d\zeta)\right),
\]
where $\omega$ is the normalized Haar measure. The map $\mathbb{B}_{\hat{\mathbf{c}}}$ is denoted by $\mathbb{B}_{a,b}.$ It is not hard to see that $\mathbb{B}_{a,b}(\mu)$ is characterized by 
\begin{equation}\label{eq0122}
\eta_{\mathbb{B}_{a,b}(\mu)}(z) = c^{-1}\eta_\mu(cz). 
\end{equation}
which is the analog of  (\ref{eq0121}) for the multiplicative case.  Moreover, for any pair of measures $\mu,\nu\in\mathcal{P}(\tor)$ and any $a \in \R,b \geq0$, we have the homomorphism properties 
\begin{align}
&\mathbb{B}_{a,b}(\mu \calt\nu) = \mathbb{B}_{a,b}(\mu)\calt\mathbb{B}_{a,b}(\nu), \\
&\mathbb{B}_{a,b}(\mu \boxtimes \nu) = \mathbb{B}_{a,b}(\mu) \boxtimes(\mathbb{B}_{a,b}(\nu), \\
&\mathbb{B}_{a,b}(\mu \utimes \nu) = \mathbb{B}_{a,b}(\mu) \utimes \mathbb{B}_{a,b}(\nu), 
\end{align}
where $\utimes$ is the multiplicative Boolean convolution (see \cite{F09a,F09b}). 
%However, unfortunately, there is no direct analog of Theorem \ref{thm101}.
\end{remark}

\end{document}